\documentclass[a4paper]{amsart}
\usepackage{amsthm} 
\usepackage{amsmath}
\usepackage{amssymb}
\usepackage{amsfonts} 
\usepackage{latexsym}
\usepackage{mathrsfs} 
\usepackage[all]{xy} \SelectTips{eu}{} 
\usepackage{hyperref}

\theoremstyle{plain} 
\newtheorem{theorem}{Theorem}
\newtheorem*{theorem*}{Theorem}
\newtheorem{proposition}[theorem]{Proposition}
\newtheorem*{proposition*}{Proposition}
\newtheorem{corollary}[theorem]{Corollary}
\newtheorem*{corollary*}{Corollary}
\newtheorem{lemma}[theorem]{Lemma}

\numberwithin{theorem}{section}

\theoremstyle{definition}
\newtheorem{definition}[theorem]{Definition}
\newtheorem{remark}[theorem]{Remark}
\newtheorem{example}[theorem]{Example} 
\swapnumbers
\newtheorem{numpar1}[theorem]{$\kappa$-pure morphisms}
\newtheorem{numpar2}[theorem]{Complexes}
\newtheorem{numpar3}[theorem]{(Pre)covers}
\newtheorem{numpar4}[theorem]{Orthogonal classes and cotorsion pairs}
\newtheorem{numpar5}[theorem]{Kaplansky classes}
\newtheorem{numpar6}[theorem]{Kaplansky classes and filtrations}

\newcommand{\Coker}[1]{\nobreak{\operatorname{Coker}#1}}
\newcommand{\ZZ}{\mathbb{Z}}
\newcommand{\pd}[2][R]{\operatorname{pd}_{#1}#2}
\newcommand{\Spec}[1]{\operatorname{Spec}#1}
\newcommand{\setof}[3][\mspace{1mu}]{\{#1#2 \mid #3#1\}}
\newcommand{\Bo}[2][]{\operatorname{B}_{#1}(#2)}
\newcommand{\Cy}[2][]{\operatorname{Z}_{#1}(#2)}
\newcommand{\lra}{\longrightarrow}
\newcommand{\xra}[2][]{\xrightarrow[#1]{\;#2\;}}
\newcommand{\deq}{\:=\:}
\newcommand{\dis}{\:\is\:}
\newcommand{\is}{\cong}
\renewcommand{\a}{\alpha}
\newcommand{\g}{\gamma}
\newcommand{\f}{\varphi}
\newcommand{\s}{\sigma}
\newcommand{\mapdef}[4][\rightarrow]{\nobreak{#2\colon #3 #1 #4}}
\newcommand{\tp}[3][R]{\nobreak{#2\otimes_{#1}#3}}
\newcommand{\ProjR}{\operatorname{Proj}R}
\newcommand{\FlatR}{\operatorname{Flat}R}
\newcommand{\FlatX}{\operatorname{Flat}X}
\newcommand{\Kflat}{\mathbf{K}(\FlatX)}
\newcommand{\Ktot}{\mathbf{K}_{{\rm tac}}(\FlatX)}
\newcommand{\Ntot}{\mathbf{F}_{{\rm tac}}(\FlatX)}
\newcommand{\Filt}{\operatorname{Filt}}
\newcommand{\calC}{\mathcal{C}}
\newcommand{\calF}{\mathcal{F}}
\newcommand{\calG}{\mathcal{G}}
\newcommand{\calO}{\mathcal{O}}
\newcommand{\calR}{\mathcal{R}}
\newcommand{\calS}{\mathcal{S}}
\newcommand{\calU}{\mathcal{U}}
\newcommand{\Scx}[2]{\operatorname{S}_{#1}(#2)}
\newcommand{\Dcx}[2]{\operatorname{D}_{#1}(#2)}
\newcommand{\Ch}[1]{\operatorname{Ch}(#1)}
\newcommand{\Id}[1]{\operatorname{Id}^#1} 
\newcommand{\Hom}{\operatorname{Hom}} 
\newcommand{\Ext}[4]{\operatorname{Ext}^{#1}_{#2}(#3,#4)}
\newcommand{\Mor}{\operatorname{Mor}} 
\newcommand{\F}{\mathscr{F}} 
\newcommand{\G}{\mathscr{G}} 
\renewcommand{\H}{\mathscr{H}} 
\renewcommand{\L}{\mathscr{L}} 
 
\newcommand{\M}{\mathscr{M}} 
\newcommand{\J}{\mathscr{J}} 
\newcommand{\N}{\mathscr{N}} 
\renewcommand{\P}{\mathscr{P}} 
\newcommand{\R}{\mathcal{O}_X} 
\newcommand{\fP}{\mathfrak{P}}
\newcommand{\Qcoh}{\mathfrak{Qcoh}} 
\newcommand{\cx}[1]{#1_\bullet}

\def\urltilda{\kern -.15em\lower .7ex\hbox{\~{}}\kern .04em} 

\def\soft#1{\leavevmode\setbox0=\hbox{h}\dimen7=\ht0\advance \dimen7
  by-1ex\relax\if t#1\relax\rlap{\raise.6\dimen7
  \hbox{\kern.3ex\char'47}}#1\relax\else\if T#1\relax
  \rlap{\raise.5\dimen7\hbox{\kern1.3ex\char'47}}#1\relax \else\if
  d#1\relax\rlap{\raise.5\dimen7\hbox{\kern.9ex \char'47}}#1\relax\else\if
  D#1\relax\rlap{\raise.5\dimen7 \hbox{\kern1.4ex\char'47}}#1\relax\else\if
  l#1\relax \rlap{\raise.5\dimen7\hbox{\kern.4ex\char'47}}#1\relax \else\if
  L#1\relax\rlap{\raise.5\dimen7\hbox{\kern.7ex
  \char'47}}#1\relax\else\message{accent \string\soft \space #1 not
  defined!}#1\relax\fi\fi\fi\fi\fi\fi}

\begin{document}

\title[A Zariski-local notion of \textbf{F}-total acyclicity]{A
  Zariski-local notion of \textbf{F}-total acyclicity\\ for complexes
  of sheaves}

\author[L.W. Christensen]{Lars Winther Christensen}

\address{L.W.C. \ Texas Tech University, Lubbock, TX 79409, U.S.A.}

\email{lars.w.christensen@ttu.edu}

\urladdr{http://www.math.ttu.edu/\urltilda lchriste}

\author[S. Estrada]{Sergio Estrada}

\address{S.E. \ Universidad de Murcia, Murcia 30100, Spain}

\email{sestrada@um.es}

\author[A. Iacob]{Alina Iacob}

\address{A.I. \ Georgia Southern University, Statesboro, GA 30460,
  U.S.A.}

\email{aiacob@georgiasouthern.edu}

\thanks{This research was conducted in spring 2015 during visits to
  the Centre de Recerca Matem\`atica, and during visits by L.W.C.\ and
  A.I.\ to the University of Murcia. The hospitality and support of
  both institutions is acknowledged with gratitude. L.W.C.\ was partly
  supported by NSA grant H98230-14-0140, and S.E.\ was partly
  supported by grant MTM2013-46837-P and FEDER funds. S.E.\ and A.I.\ are also supported by the grant 18394/JLI/13 by the Fundaci\'on S\'eneca-Agencia de Ciencia y Tecnolog\'{\i}a de la Regi\'on de Murcia in the framework of III PCTRM 2011-2014.}

\date{\today}



\begin{abstract}
  We study a notion of total acyclicity for complexes of flat sheaves
  over a scheme. It is Zariski-local---i.e.\ it can be verified on any
  open affine covering of the scheme---and for sheaves over a
  quasi-compact semi-separated scheme it agrees with the categorical
  notion. In particular, it agrees, in their setting,
  with the notion studied by Murfet and Salarian for sheaves over a
  noetherian semi-separated scheme. As part of the study we recover,
  and in several cases extend the validity of, recent results on
  existence of covers and precovers in categories of sheaves. One
  consequence is the existence of an adjoint to the inclusion of these
  totally acyclic complexes into the homotopy category of complexes of
  flat sheaves.
\end{abstract}

\maketitle

\section*{Introduction}
\noindent
This paper is part of a thrust to extend Gorenstein homological
algebra to schemes. The first major advance was made by Murfet and
Salarian \cite{MS}, who introduced an operational notion of total
acyclicity over noetherian semi-separated schemes.  Total acyclicity
has its origin in Tate cohomology of finite group representations,
which is computed via, what we now call, totally acyclic complexes of
projectives. The contemporary terminology was introduced in works of
Avramov and Martsinkovsky \cite{AM} and Veliche \cite{Vel}: Given a
commutative ring $R$, a chain complex $\cx{P}$ of projective
$R$-modules is called \emph{totally acyclic} if it is acyclic and
$\Hom_R(\cx{P},Q)$ is acyclic for every projective $R$-module $Q$.

Categories of sheaves do not, in general, have enough projectives, so
it is not obvious how to define an interesting notion of total
acyclicity in this setting. Murfet and Salarian's approach was to
focus on flat sheaves: They say that a complex $\cx{\F}$ of flat
quasi-coherent sheaves over a noetherian semi-separated scheme $X$ is
\emph{\textbf{F}-totally acyclic} if it is acyclic and
$\tp[]{\mathscr{I}}{\cx{\F}}$ is acyclic for every injective
quasi-coherent sheaf $\mathscr{I}$ on $X$. This notion has its origin
in the work of Enochs, Jenda, and Torrecillas, who in \cite{EJT}
introduced it for complexes of modules. The assumptions on $X$ ensure
that a quasi-coherent sheaf on $X$ is (categorically) injective if and
only if every section of the sheaf is an injective module. In fact,
\textbf{F}-total acyclicity as defined in \cite{MS} is a
\emph{Zariski-local} property. That is, a complex $\cx{\F}$ of flat
quasi-coherent sheaves on $X$ is \textbf{F}-totally acyclic if there
is an open affine covering $\calU$ of $X$ such that $\cx{\F}(U)$ is an
\textbf{F}-totally acyclic complex of flat modules for every
$U\in\calU$.

In this paper we give a definition of \textbf{F}-total acyclicity for
complexes of flat quasi-coherent sheaves without placing any
assumptions on the underlying scheme. Our definition is Zariski-local,
and it agrees with the one from \cite{MS} when the latter applies. In
fact we prove more, namely (Proposition \ref{prp:equiv_qcss}) that
over any quasi-compact semi-separated scheme $X$, an acyclic complex
$\cx{\F}$ of flat quasi-coherent sheaves is \textbf{F}-totally acyclic
per our definition if and only if $\tp[]{\mathscr{I}}{\cx{\F}}$ is
acyclic for every injective quasi-coherent sheaf $\mathscr{I}$ on $X$.

The key to the proof of Zariski-localness (Corollary
\ref{N-tot.acyc.def}) is the next result (Proposition
\ref{prp:N-tot.local}), which says that \textbf{F}-total acyclicity
for complexes of modules is a so-called \emph{ascent--descent}
property. By a standard argument (Lemma \ref{faithflat}), this implies
that the corresponding property of complexes of quasi-coherent sheaves
is Zariski-local.

\begin{proposition*}
  Let $\varphi\colon R\to S$ be a flat homomorphism of commutative
  rings.
  \begin{enumerate}
  \item If $\cx{F}$ is an \textbf{F}-totally acyclic complex of flat
    $R$-modules, then $\tp{S}{\cx{F}}$ is an \textbf{F}-totally
    acyclic complex of flat $S$-modules.
  \item If $\varphi$ is faithfully flat and $\cx{F}$ is a complex of
    $R$-modules such that $\tp{S}{\cx{F}}$ is an \textbf{F}-totally
    acyclic complex of flat $S$-modules, then $\cx{F}$ is an
    \textbf{F}-totally acyclic complex of flat $R$-modules.
  \end{enumerate}
\end{proposition*}

In lieu of projective sheaves one can focus on vector bundles---not
necessarily finite dimensional. In Section \ref{sec:totNtot} we touch
on a notion of total acyclicity for complexes of vector bundles. By
comparing it to \textbf{F}-total acyclicity we prove that it is
Zariski-local for locally coherent locally $d$-perfect schemes
(Theorems~\ref{totac.is.Ntotac}--\ref{tot.acyclic.mod}).

The keystone result of Section \ref{sec:adj} is Theorem
\ref{Ntot.complete}, which says that for any scheme $X$, the class of
\textbf{F}-totally acyclic complexes is covering in the category of
chain complexes of quasi-coherent sheaves on $X$. It has several
interesting consequences. The homotopy category of chain complexes of
flat quasi-coherent sheaves over $X$ is denoted $\Kflat$. It is a
triangulated category, and the full subcategory $\Ktot$ of
\textbf{F}-totally acyclic complexes in $\Kflat$ is a triangulated
subcategory. Theorem \ref{Ntot.complete} allows us to remove
assumptions on the scheme in \cite[Corollary 4.26]{MS} and obtain
(Corollary \ref{cor.right.adj}):

\begin{corollary*}
  For any scheme $X$, the inclusion $\Ktot\to \Kflat$ has a right
  adjoint.
\end{corollary*}

A \emph{Gorenstein flat} quasi-coherent sheaf is defined as a cycle
sheaf in an \textbf{F}-totally acyclic complex of flat quasi-coherent
sheaves. Enochs and Estrada \cite{EE} prove that every quasi-coherent
sheaf over any scheme has a flat precover. As another consequence
(Corollary \ref{Gor.flat.prec}) of Theorem \ref{Ntot.complete} we
obtain a Gorenstein version of this result; the affine case was
already proved by Yang and Liang \cite{yang14}.

\begin{corollary*}
  Let $X$ be a scheme. Every quasi-coherent sheaf on $X$ has a
  Gorenstein flat precover. If $X$ is quasi-compact and
  semi-separated, then the Gorenstein flat precover is an epimorphism.
\end{corollary*}

Finally, Theorem \ref{Ntot.complete} combines with Theorem
\ref{totac.is.Ntotac} to yield (Corollary \ref{cor.Gor.prec}):

\begin{corollary*}
  Let $R$ be a commutative coherent $d$-perfect ring. Every $R$-module
  has a Gorenstein projective precover.
\end{corollary*}

\noindent This partly recovers results of Estrada, Iacob, and
Odaba\c{s}\i\ \cite[Corollary~2]{EIO} and of Bravo, Gillespie, and
Hovey \cite[Proposition~8.10]{BGH}.

\section{Preliminaries}

\noindent
Let $\kappa$ be a cardinal, by which we shall always mean a regular
cardinal. An object $A$ in a category $\calC$ is called
\emph{$\kappa$-presentable} if the functor $\Hom_{\calC}(A,-)$
preserves $\kappa$-directed colimits. A category $\calC$ is called
\emph{locally $\kappa$-presentable} if it is cocomplete and there is a
set $\calS$ of $\kappa$-presentable objects in $\calC$ such that every
object in $\calC$ is a $\kappa$-directed colimit of objects in
$\calS$.

An $\aleph_0$-directed colimit is for short called a \emph{direct
  limit.}

\begin{numpar1}
  \label{ss:pure}
  Let $\kappa$ be a cardinal and $\calC$ be a category. We recall from
  the book of Ad{\'a}mek and J. Rosick{\'y} \cite[Definition 2.27]{AR}
  that a morphism $\mapdef{\f}{A}{B}$ in $\calC$ is called
  \emph{$\kappa$-pure} if for every commutative square
  \begin{equation*}
    \xymatrix{
      A'\ar[r]^-{\f'} \ar[d]^-\a & B'\ar[d] \\
      A\ar[r]^-{\f} & B 
    }
  \end{equation*}
  where the objects $A'$ and $B'$ are $\kappa$-presentable, there
  exits a morphism $\mapdef{\g}{B'}{A}$ with $\g\circ\f' = \a$. A
  subobject $A \subseteq B$ is called \emph{$\kappa$-pure} if the
  monomorphism $A \to B$ is $\kappa$-pure.
\end{numpar1}

\begin{numpar2}
  \label{ss:complexes}
  In the balance of this section, $\calG$ denotes a Grothendieck
  category, and $\Ch{\calG}$ denotes the category of chain complexes
  over $\calG$. It is elementary to verify that $\Ch{\calG}$ is
  also a Grothendieck category. We use homological notation, so a
  complex $\cx{M}$ in $\Ch{\calG}$ looks like this
  \begin{equation*}
    \cx{M} \deq \cdots\lra M_{i+1} \xra{\partial_{i+1}}  
    M_i \xra{\partial_i} M_{i-1} \lra \cdots 
  \end{equation*}
  We denote by $\Cy[n]{\cx{M}}$ and $\Bo[n]{\cx{M}}$ the $n$th cycle
  and $n$th boundary object of $\cx{M}$.

  Let $M$ be an object in $\calG$. We denote by $\Scx{n}{M}$ the
  complex with $M$ in degree $n$ and $0$ elsewhere. By $\Dcx{n}{M}$ we
  denote the complex with $M$ in degrees $n$ and $n-1$, differential
  $\partial_n= \Id{M}$, the identity map, and $0$ elsewhere.
\end{numpar2}

\begin{numpar3}
  Let $\calF$ be a class of objects in $\calG$. A morphism
  $\mapdef{\phi}{F}{M}$ in $\calG$ is called an
  \emph{$\calF$-precover} if $F$ is in $\calF$
  and $$\Hom_{\calG}(F',F) \lra \Hom_{\calG} (F',M) \lra 0$$ is exact
  for every $F' \in \calF$. Further, if $\mapdef{\phi}{F}{M}$ is a
  precover and every morphism $\mapdef{\s}{F}{F}$ with $\phi\s = \phi$
  is an automorphism, then $\phi$ is called an \emph{$\calF$-cover.}
  If every object in $\calG$ has an $\calF$-(pre)cover, then the class
  $\calF$ is called \emph{(pre)covering}.
\end{numpar3}
The dual notions are \emph{(pre)envelope} and \emph{(pre)enveloping.}

\begin{numpar4}
  Let $\calF$ be a class of objects in $\calG$ and consider the
  orthogonal classes
  \begin{align*}
    \calF^\perp &\deq \setof{G\in \calG}
    {\Ext{1}{\calG}{F}{G}=0\ \text{for all } F\in \calF}\ \text{ and}\\
    ^{\perp}\calF &\deq \setof{G\in \calG} {\Ext{1}{\calG}{G}{F}=0\
      \text{for all } F\in \calF}\:.
  \end{align*}
  Let $\calS \subseteq \calF$ be a set. The pair
  $(\calF,\calF^{\perp})$ is said to be \emph{cogenerated by the set
    $\calS$} if an object $G$ belongs to $\calF^\perp $ if and only if
  $\Ext{1}{\calG}{F}{G}=0$ holds for all $F\in \calS$.

  A pair $(\calF,\calC)$ of classes in $\calG$ with
  $\calF^{\perp}=\calC$ and $^{\perp}\calC=\calF$ is called a
  \emph{cotorsion pair}.  A cotorsion pair $(\calF,\calC)$ in $\calG$
  is called \emph{complete} provided that for every $M\in \calG$ there
  are short exact sequences $0\to C\to F\to M\to 0$ and $0\to M\to
  C'\to F'\to 0$ with $F,F'\in \calF$ and $C,C'\in \calC$. Notice that
  for every complete cotorsion pair $(\calF,\calC)$, the class $\calF$
  is precovering and the class $\calC$ is preenveloping.
\end{numpar4}

\begin{numpar5}
  Let $\calF$ be a class of objects in $\calG$ and $\kappa$ be a
  cardinal. One says that $\calF$ is a \emph{$\kappa$-Kaplansky class}
  if for every inclusion $Z\subseteq F$ of objects in $\calG$ such
  that $F$ is in $\calF$ and $Z$ is $\kappa$-presentable, there exists
  a $\kappa$-presentable object $W$ in $\calF$ with $Z\subseteq
  W\subseteq F$ and such that $F/W$ belongs to $\calF$. We say that
  $\calF$ is a \emph{Kaplansky class} if it is a $\kappa$-Kaplansky
  for some cardinal $\kappa$.
\end{numpar5}

\begin{proposition}
  \label{Kap.covering}
  Every Kaplansky class in $\calG$ that is closed under extensions and
  direct limits is covering.
\end{proposition}

\begin{proof}
  Let $\kappa$ be a cardinal; let $\calF$ be a $\kappa$-Kaplansky
  class in $\calG$ and assume that it is closed under extensions and
  direct limits.  It now follows from Eklof's lemma \cite[Lemma~1]{ET}
  that the pair $(\calF,\calF^\perp)$ is cogenerated by a set.

  Let $M$ be an object in $\calG$. Denote by $\widetilde{M}$ the sum
  of all images in $M$ of morphisms with domain in $\calF$. That is,
  $\widetilde{M} = \sum_{\varphi\in \Hom(F,M),\,F\in \calF}{\rm
    Im}(\varphi)$; as $\calG$ is well-powered the sum is
  well-defined. Since $\calF$ is closed under coproducts, there exists
  a short exact sequence $0\to L\to E\to \widetilde{M}\to 0$, with
  $E\in \calF$. By a result of Enochs, Estrada, Garc\'{\i}a Rozas, and
  Oyonarte \cite[Theorem~2.5]{EEGO} there is a short exact sequence
  $0\to L\to C\to F\to 0$ with $C\in \calF^{\perp}$ and $F\in \calF$.
  Consider the push-out diagram,
  \begin{equation*}
    \xymatrix{
      & 0\ar[d] & 0\ar[d] &  &\\
      0\ar[r] & L\ar[d]\ar[r] & E \ar[r] \ar[d]\ar[r] &\widetilde{M}\ar@{=}[d]\ar[r] &0\\
      0\ar[r] & C\ar[d]\ar[r] & D\ar[d] \ar[r] &\widetilde{M}\ar[r] &0 \\
      & F\ar@{=}[r]\ar[d] & F\ar[d] & &\\& 0 & 0 & & 
    }
  \end{equation*}
  Since $\calF$ is closed under extensions, one has $D\in \calF$,
  and since $C\in \calF^{\perp}$ it follows that $D \to \widetilde{M}$
  is an $\calF$-precover. By the definition of $\widetilde{M}$, it
  immediately follows that $M$ has an $\calF$-precover, so the class
  $\calF$ is precovering. Finally, any precovering class that is
  closed under direct limits is covering; see Xu \cite[Theorem
  2.2.12]{Xu} for an argument in a module category that carries over
  to Grothendieck categories.
\end{proof}

\begin{numpar6}
  Recall that a well ordered direct system,
  $\setof{M_{\alpha}}{\alpha\le \lambda}$, of objects in $\calG$ is
  called \emph{continuous} if one has $M_0=0$ and, for each limit
  ordinal $\beta\leq \lambda$, one has $M_{\beta} =
  \varinjlim_{\alpha<\beta } M_{\alpha}$. If all morphisms in the
  system are monomorphisms, then the system is called a
  \emph{continuous directed union}.

  Let $\calS$ be a class of objects in $\calG$. An object $M$ in
  $\calG$ is called \emph{$\calS$-filtered} if there is a continuous
  directed union $\setof{M_{\alpha}}{ \alpha\le \lambda}$ of
  subobjects of $M$ such that $M = M_{\lambda}$ and for every
  $\alpha<\lambda$ the quotient $M_{\alpha+1}/M_{\alpha}$ is
  isomorphic to an object in $\calS$. We denote by $\Filt(\calS)$ the
  class of all $\calS$-filtered objects in $\calG$.

  Let $\kappa$ be a cardinal and $\calF$ be a $\kappa$-Kaplansky class
  in $\calG$ that is closed under direct limits. It is standard to
  verify that there exists a set $\calS$ of $\kappa$-presentable
  objects in $\calF$ with $\calF\subseteq \Filt(\calS)$; see for
  example the proof of \cite[Lemma~4.3]{G}. In general the classes
  $\calF$ and $\Filt(\calS)$ need not be equal, but if $\calF$ is
  closed under extensions, then equality holds. An explicit example of
  strict containment is provided by the (Kaplansky) class of phantom
  morphisms in the (Grothendieck) category of representations of the
  $\mathbb{A}_2$ quiver; see Estrada, Guil Asensio, and
  Ozbek~\cite{EGO}.

  \v{S}\soft{t}ov{\'{\i}}{\v{c}}ek proves in \cite[Corollary
  2.7(2)]{Sto} that every Kaplansky class $\calF$ that is closed under
  direct limits (and extensions) is \emph{deconstructible,} which per
  \cite[Definition 1.4]{Sto} means precisely that there exists a set
  $\calS$ with $\calF = \Filt(\calS)$. However, the assumption about
  closedness under extensions is not stated explicitly.
\end{numpar6}

\section{Faithfully flat descent for \textbf{F}-total acyclicity}
\label{sec:Ntot}

\begin{definition}
  \label{localprop}
  Let $X$ be a scheme with structure sheaf $\R$, and let $\fP$ be a
  property of modules over commutative rings.

  (1) A quasi-coherent sheaf $\M$ on $X$ is said to \emph{locally}
  have property $\fP$ if for every open affine subset $U\subseteq X$,
  the $\R(U)$-module $\M(U)$ has property $\fP$.

  (2) As a (local) property of quasi-coherent sheaves on $X$, the
  property $\fP$ is called \emph{Zariski-local} if the following
  conditions are equivalent for every quasi-coherent sheaf $\M$ on
  $X$.
  \begin{itemize}
  \item The sheaf $\M$ locally has property $\fP$.
  \item There exists an open affine covering $\calU$ of $X$ such that
    for every $U \in\calU$ the $\R(U)$-module $\M(U)$ has property
    $\fP$.
  \end{itemize}
\end{definition}

That is, Zariski-localness of a property of sheaves means that it can
be verified on any open affine covering. A useful classic tool for
verifying Zariski-localness is based on flat ascent and descent of the
underlying module property. We make it explicit in
Lemma~\ref{faithflat}; see also \cite[\S34.11]{SP}.

\begin{definition}
  \label{ad-property}
  Let $\fP$ be a property of modules over commutative rings and let
  $\calR$ be a class of commutative rings.
  \begin{enumerate}
  \item $\fP$ is said to \emph{ascend} in $\calR$, if for every flat
    epimorphism $R\to S$ of rings in $\calR$ and for every $R$-module
    $M$ with property $\fP$, the $S$-module $\tp{S}{M}$ has property
    $\fP$.
  \item $\fP$ is said to \emph{descend} in $\calR$ if an $R$-module
    $M$ has property $\fP$ whenever there exists a faithfully flat
    homomorphism $R\to S$ of rings in $\calR$ such that the $S$-module
    $\tp{S}{M}$ has property $\fP$.
  \end{enumerate}

  If $\fP$ ascends and descends in $\calR$, then it is called an
  \emph{ascent--descent property}, for short an \emph{AD-property}, in
  $\calR$.
\end{definition}

\begin{definition}
  \label{dfn:compat}
  A property $\fP$ of modules over commutative rings is said to be
  \emph{compatible with finite products} if the following conditions
  are equivalent for all commutative rings $R_1$ and $R_2$, all
  $R_1$-modules $M_1$, and all $R_2$-modules $M_2$.
  \begin{itemize}
  \item $M_1$ and $M_2$ have property $\fP$.
  \item The $R_1\times R_2$-module $M_1\times M_2$ has property $\fP$.
  \end{itemize}
\end{definition}

\begin{lemma}
  \label{faithflat}
  Let $X$ be a scheme with structure sheaf $\R$ and let $\fP$ be a
  property of modules over commutative rings. If $\fP$ is compatible
  with finite products and an AD--property in the class $ \calR =
  \setof{\R(U)}{U \subseteq X \text{ is an open affine subset}}$ of
  commutative rings, then $\fP$ as a property of quasi-coherent
  sheaves on $X$ is Zariski-local.
\end{lemma}

\begin{proof} 
  Let $\calU =\setof{U_i}{i\in I}$ be an open affine covering of $X$
  such that for every $i\in I$ the $\R(U_i)$-module $\M(U_i)$ has
  property $\fP$. Let $U$ be an arbitrary open affine subset of $X$.
  There exists a standard open covering $U = \bigcup_{j=1}^n D(f_j)$
  such that for every $j$ there is a $U_j \in \calU$ with the property
  that $D(f_j)$ is a standard open subset of $U_j$; that is,
  $f_j\in\R(U_j)$ and $D(f_j) = \Spec (\R (U_j)_{f_j})$; see
  \cite[Lemma 25.11.5]{SP}. In particular, one has $\M(D(f_j)) \is
  \M(U_j) \otimes_{\R(U_j)} \R(U_j) _{f_j}$, and it follows that
  $\M(D(f_j))$ has property $\fP$ as it ascends in $\calR$. The
  compatibility of $\fP$ with direct products now ensures that the
  module $\prod_{j=1}^n \M(D(f_j))$ over $\prod_{j=1}^n \R(D(f_j))$
  has property $\fP$. As the canonical morphism $\R(U) \to
  \prod_{j=1}^n \R(D(f_j))$ is faithfully flat, it follows that the
  $\R(U)$-module $\M(U)$ has property $\fP$.
\end{proof}

While ascent of a module property is usually easy to prove, it can be
more involved to prove descent.  For instance, it is easy to see that
flatness is an AD-property. Also the flat Mittag-Leffler property is
known to be an AD-property: Ascent is easy to prove, while descent
follows from Raynaud and Gruson \cite[II.5.2]{RG}; see also Perry
\cite[\S9]{P} for correction of an error in \cite{RG}. The AD-property
is also satisfied by the $\kappa$-restricted flat Mittag-Leffler
modules, where $\kappa$ is an infinite cardinal (see Estrada, Guil
Asensio, and Trlifaj~\cite{EGT}).

We are also concerned with properties of complexes of sheaves and
modules; it is straightforward to extend Definitions~\ref{localprop}--\ref{dfn:compat}
and Lemma~\ref{faithflat} to the case where $\fP$ is a property of
complexes.

Next we introduce the property for which we will study
Zariski-localness.

\begin{definition}
  A complex of flat $R$-modules $\cx{F}= \cdots\to F_{i+1}\to F_i\to
  F_{i-1}\to\cdots$ is called \emph{\textbf{F}-totally acyclic} if it
  is acyclic and $\tp{I}{\cx{F}}$ is acyclic for every injective
  $R$-module~$I$.
\end{definition}

\begin{definition}
  \label{def:N-tot.acyc.def}
  Let $X$ be a scheme with structure sheaf $\R$. A complex $\cx{\F} =
  \cdots\to \F_{i+1}\to \F_i\to \F_{i-1}\to \cdots$ of flat
  quasi-coherent sheaves on $X$ is called \emph{\textbf{F}-totally
    acyclic} if for every open affine subset $U\subseteq X$ the
  complex $\cx{\F}(U)$ of flat $\R(U)$-modules is \textbf{F}-totally
  acyclic.
\end{definition}

The next lemma shows, in particular, that \textbf{F}-total acyclicity
is an AD-property.

\begin{proposition}
  \label{prp:N-tot.local}
  Let $\varphi\colon R\to S$ be a flat homomorphism of commutative
  rings.
  \begin{enumerate}
  \item If $\cx{F}$ is an \textbf{F}-totally acyclic complex of flat
    $R$-modules, then $\tp{S}{\cx{F}}$ is an \textbf{F}-totally
    acyclic complex of flat $S$-modules.
  \item If $\varphi$ is faithfully flat and $\cx{F}$ is a complex of
    $R$-modules such that $\tp{S}{\cx{F}}$ is an \textbf{F}-totally
    acyclic complex of flat $S$-modules, then $\cx{F}$ is an
    \textbf{F}-totally acyclic complex of flat $R$-modules.
  \end{enumerate}
\end{proposition}

\begin{proof}
  (1) Since $\varphi$ is flat and $\cx{F}$ is an acyclic complex of
  flat $R$-modules, it follows that $\tp{S}{\cx{F}}$ is an acyclic
  complex of flat $S$-modules. Now, let $E$ be an injective
  $S$-module, by flatness of $\varphi$ it is also injective as an
  $R$-module. Indeed, there are isomorphisms $\Hom_R(-,E) \is
  \Hom_R(-,\Hom_S(S,E)) \is \Hom_S( S\otimes_R -, E)$. It follows that
  $\tp[S]{E}{(\tp{S}{\cx{F}})} \is \tp{E}{\cx{F}}$ is acyclic.

  (2) Since $\varphi$ is faithfully flat and $\tp{S}{\cx{F}}$ is an
  \textbf{F}-totally acyclic complex of flat $S$-modules, it follows
  that $\cx{F}$ is an acyclic complex of flat $R$-modules. Let $I$ be
  an injective $R$-module; it must be shown that $\tp{I}{\cx{F}}$ is
  acyclic. The $S$-module $\Hom_R(S,I)$ is injective, so it follows
  from the next chain of isomorphisms that $\tp{\Hom_R(S,I)}{\cx{F}}$
  is acyclic:
  \begin{equation*}
    \tp{\Hom_R(S,I)}{\cx{F}} \dis \tp{(\tp[S]{\Hom_R(S,I)}{S})}{\cx{F}} 
    \dis \tp[S]{\Hom_R(S,I)}{(\tp{S}{\cx{F}})}\:;
  \end{equation*}
  here the last complex is acyclic by the assumption that
  $\tp{S}{\cx{F}}$ is \textbf{F}-totally acyclic.

  As $\varphi$ is faithfully flat, the exact sequence of $R$-modules
  $0\to R\to S\to S/R\to 0$ is pure. Hence the induced sequence
  \begin{equation*}
    0\lra \Hom_R(S/R,I)\lra \Hom_R(S,I)\lra \Hom_R(R,I)\lra 0
  \end{equation*}
  is split exact. It follows that $\Hom_R(R,I)\is I$ is a direct
  summand of $\Hom_R(S,I)$ as an $R$-module. Hence the complex
  $\tp{I}{\cx{F}}$ is a direct summand of the acyclic complex
  $\tp{\Hom_R(S,I)}{\cx{F}}$ and, therefore, acyclic.
\end{proof}

The proposition above together with (the complex version of)
Lemma~\ref{faithflat} ensure that the property of being
\textbf{F}-totally acyclic is Zariski-local as a property of complexes
of quasi-coherent sheaves.

\begin{corollary}
  \label{N-tot.acyc.def}
  Let $X$ be a scheme with structure sheaf $\R$. A complex $\cx{\F}$
  of flat quasi-coherent sheaves on $X$ is \textbf{F}-totally acyclic
  if there exists an affine open covering $\calU$ of $X$ such that the
  complex $\cx{\F}(U)$ of flat $\R(U)$-modules is \textbf{F}-totally
  acyclic for every $U\in \calU$.\qed
\end{corollary}

\begin{remark}
  Our definition \ref{def:N-tot.acyc.def} is different from
  \cite[Definition 4.1]{MS}, but as shown in \cite[Lemma 4.5]{MS} the
  two are equivalent if $X$ is Noetherian and semi-separated, which is
  the blanket assumption in \cite{MS}. In the next proposition we
  substantially relax the hypothesis on $X$ and show that our
  definition coincides with the one from \cite{MS} if $X$ is
  quasi-compact and semi-separated.
\end{remark}

\begin{proposition}
  \label{prp:equiv_qcss}
  Let $X$ be a quasi-compact and semi-separated scheme and let
  $\cx{\F}$ be an acyclic complex of flat quasi-coherent sheaves on $X$.
  Conditions $(i)$ and $(ii)$ below are equivalent and imply $(iii)$.
  \begin{enumerate}
  \item[($i$)] The complex $\J\otimes_{\R}\cx{\F}$ is acyclic for
    every injective object $\J$ in $\Qcoh(X)$.
  \item[($ii$)] There exists a semi-separating open affine covering
    $\calU$ of $X$ such that for every $U\in \calU$, the complex
    $\cx{\F}(U)$ of flat $\R(U)$-modules is \textbf{F}-totally
    acyclic.
  \item[($iii$)] For every $x\in X$ the complex $(\cx{\F})_x$ of flat
    $\calO_{X,x}$-modules is \textbf{F}-totally acyclic.
  \end{enumerate}
\end{proposition}

\begin{proof}
  Let $\calU$ be a semi-separating open affine covering of $X$. For
  every $U\in \calU$, the inclusion $U\to X$ gives an adjoint pair
  $(j^*_U,j_*^U)$, where $j^*_U\colon\Qcoh(X)\to \Qcoh(U)$ and
  $j_*^U:\Qcoh(U)\to \Qcoh(X)$ are the inverse and direct image functor
  respectively.  Since $j^*_U$ is an exact functor and $j_*^U$ is a
  right adjoint of $j^*_U$, it follows that $j_*^U$ preserves
  injective objects. Now the implication $(i)\Rightarrow (ii)$
  follows.

  $(ii)\Rightarrow (i)$: Let $\calU$ be a semi-separating affine
  covering of $X$, such that $\cx{\F}(U)$ is \textbf{F}-totally
  acyclic for every $U\in \calU$. Without loss of generality, assume
  that $\calU$ is finite. Given a quasi-coherent sheaf $\J$ there
  exists, since the scheme is semi-separated, a monomorphism
  \begin{equation*}
    0 \lra \J \lra \prod_{U\in\calU} j^U_*(\widetilde{E_U})\:,
  \end{equation*}
  where $E_U$ denotes the injective hull of $\J(U)$ in the category of
  $\R(U)$-modules, and $\widetilde{E_U} \in \Qcoh(U)$ is the
  corresponding sheaf. Recall that each quasi-coherent sheaf
  $j^U_*(\widetilde{E_U})$ is injective per the argument above. We
  assume that $\J$ is injective in $\Qcoh(X)$, so it is a direct
  summand of the finite product $\prod_{\calU}
  j^U_*(\widetilde{E_U})$. It is thus sufficient to prove that
  $j^U_*(\widetilde{E_U})\otimes_{\R} \cx{\F} $ is acyclic for every
  $U\in\calU$, as that will imply that $\J \otimes_{\R}\cx{\F}$ is
  a direct summand of an acyclic complex and hence acyclic. Fix $U\in
  \calU$; for every $W \in \calU$ there are isomorphisms
  \begin{align*}
    (j^U_*(\widetilde{E_U})\otimes_{\R}\cx{\F} )(W) &\cong 
    j^U_*(\widetilde{E_U})(W)\otimes_{\R(W)} \cx{\F}(W) \\ 
    &= (E_U\otimes_{\R(U)}\R(U\cap W))\otimes_{\R(W)}\cx{\F}(W) \\
    &\cong  E_U \otimes_{\R(U)}(\R(U\cap W)\otimes_{\R(W)}\cx{\F}(W))  \\
    &\cong E_U\otimes_{\R(U)} \cx{\F}(U\cap W)\\
    &\cong   E_U\otimes_{\R(U)}( \R(W\cap U)\otimes_{\R(U)}\cx{\F}(U))\\
    &\cong (E_U \otimes_{\R(U)} \R(W\cap U))\otimes_{\R(U)}\cx{\F}(U) \\
    &\cong (\R(W\cap U) \otimes_{\R(U)}E_U )\otimes_{\R(U)}\cx{\F}(U) \\
    &\cong \R(W\cap U) \otimes_{\R(U)}(E_U\otimes_{\R(U)} \cx{\F}(U)) \:.
  \end{align*}
  The last complex is acyclic as $\R(W\cap U)$ is a flat
  $\R(U)$-module and the complex $E_U\otimes_{\R(U)} \cx{\F}(U)$ is
  acyclic by the assumption that $\cx{\F}(U)$ is \textbf{F}-totally
  acyclic.

  $(i) \Rightarrow (iii)$: Given an $x\in X$ consider a subset $U\in
  \calU$ with $x\in U$. Let $E$ be an injective $\calO_{X,x}$-module;
  it is injective over $\R(U)$ as well as one has $\calO_{X,x}\cong
  (\R(U))_x$. By $(i)$ the complex
  $j_*^U(\widetilde{E})\otimes_{\R}\cx{\F}$ is acyclic, and hence so
  is
  \begin{equation*}
    (j_*^U(\widetilde{E})\otimes_{\R}\cx{\F})_x\cong
    (E\otimes_{\calO_{X,x}}\cx{\F})_x\:.
  \end{equation*}
  Thus, the $\calO_{X,x}$-complex $(\cx{\F})_x$ is
  $\textbf{F}$-totally acyclic.
\end{proof}

If the scheme is noetherian and semi-separated, then all three
conditions in Proposition \ref{prp:equiv_qcss} are equivalent; see
\cite[Lemma 4.4]{MS} for the remaining implication.

\section{Total acyclicity vs.\ \textbf{F}-total acyclicity}
\label{sec:totNtot}

\noindent
Let $X$ be a scheme with structure sheaf $\R$.  Recall from Drinfeld
\cite[Section 2]{D} that a not necessarily finite dimensional vector
bundle on $X$ is a quasi-coherent sheaf $\P$ such that the
$\R(U)$-module $\P(U)$ is projective for every open affine subset $U
\subseteq X$; i.e.\ it is locally projective per Definition
\ref{localprop}. This is a Zariski-local notion because projectivity
of modules is an AD-property and compatible with finite products; see
\cite{P}.  We take a special interest in \textbf{F}-totally acyclic
complexes of infinite dimensional vector bundles; to this end we
recall:

\begin{definition}
  \label{dfn:tot}
  A complex $\cx{P} = \cdots\to P_{i+1}\to P_i\to P_{i-1}\to \cdots$
  of projective $R$-modules is called \emph{totally acyclic} if it is
  acyclic, and $\Hom(\cx{P},Q)$ is acyclic for every projective
  $R$-module $Q$.
\end{definition}

As opposed to \textbf{F}-total acyclicity, it is not clear to us that
total acyclicity leads to a Zariski-local property of complexes of
quasi-coherent sheaves on an arbitrary scheme. The purpose of this
section is to identify conditions on schemes that ensure that total
acyclicity coincides with \textbf{F}-total acyclicity for complexes of
vector bundles.


We need a few definitions parallel to those in Section~\ref{sec:Ntot}.

\begin{definition}
  \label{localprop-r}
  Let $\fP$ be a property of commutative rings.

  (1) A scheme $X$ with structure sheaf $\R$ is said to \emph{locally}
  have property $\fP$ if for every open affine subset $U\subseteq X$
  the ring $\R(U)$ has property $\fP$.

  (2) As a (local) property of schemes, the property $\fP$ is called
  \emph{Zariski-local} if the following conditions are equivalent for
  every scheme~$X$.
  \begin{itemize}
  \item $X$ locally has property $\fP$.
  \item There exists an open affine covering $\calU$ of $X$ such that
    for every $U \in\calU$ the ring $\R(U)$ has property $\fP$.
  \end{itemize}
\end{definition}

Recall that a commutative ring $R$ is called \emph{$d$-perfect} if
every flat $R$-module has projective dimension at most $d$. Bass
\cite[Theorem~P]{HBs60} described the $0$-perfect rings.

\begin{example}
  Every locally Noetherian scheme of Krull dimension $d$ is locally
  coherent and locally $d$-perfect.

  If $R$ is a commutative coherent ring of global dimension $2$, then
  the polynomial ring in $n$ variables over $R$ is coherent of global
  dimension $n+2$; see Glaz \cite[Theorem 7.3.14]{ccr}. Thus, the
  scheme $\mathbb{P}^n_R$ is locally coherent and locally
  $(n+2)$-perfect.
\end{example}

\begin{definition}
  \label{ad-property-ring}
  Let $\fP$ a property of commutative rings and let $\calR$ be a class
  of commutative rings.
  \begin{enumerate}
  \item $\fP$ is said to \emph{ascend} in $\calR$, if for every flat
    epimorphism $R\to S$ of rings in $\calR$, the ring $S$ has
    property $\fP$ if $R$ has property $\fP$.
  \item $\fP$ is said to \emph{descend} in $\calR$ if for every
    faithfully flat homomorphism $R\to S$ of rings in $\calR$, the
    ring $S$ has property $\fP$ only if $R$ has property $\fP$.
  \end{enumerate}

  If $\fP$ ascends and descends in $\calR$, then it is called an
  \emph{ascent--descent property}, for short an \emph{AD-property}, in
  $\calR$.
\end{definition}

\begin{definition}
  A property $\fP$ of commutative rings is said to be \emph{compatible
    with finite products} if for all commutative rings $R_1$ and
  $R_2$, the product ring $R_1\times R_2$ has property $\fP$ if and
  only if $R_1$ and $R_2$ have property $\fP$.
\end{definition}

The proof of the next lemma is parallel to that of
Lemma~\ref{faithflat}.

\begin{lemma}
  \label{faithflat-ring}
  Let $X$ be a scheme with structure sheaf $\R$ and let $\fP$ be a
  property of commutative rings. If $\fP$ is compatible with finite
  products and an AD-property in the class of commutative rings
  $\setof{\R(U)}{U \subseteq X \text{ is an open affine subset}}$,
  then $\fP$ as a property of schemes is Zariski-local. \qed
\end{lemma}

\begin{proposition}
  The properties local coherence and local $d$-perfectness of schemes
  are Zariski-local.
\end{proposition}

\begin{proof}
  Coherence and $d$-perfectness are properties of rings that are
  compatible with finite products, so by Lemma~\ref{faithflat-ring} it
  is enough to prove that they are AD-properties.

  Harris \cite[Corollary 2.1]{Harris} proves that coherence descends
  along faithfully flat homomorphisms of rings. To prove ascent, let
  $R \to S$ be a flat epimorphism and assume that $R$ is coherent. Let
  $\setof{F_i}{i\in I}$ be a family of flat $S$-modules. As $S$ is
  flat over $R$, every flat $S$-module is a flat $R$-module. Since $R$
  is coherent the $R$-module $\prod_{i\in I}F_i$ is flat, and as
  flatness ascends so is the $S$-module $\tp{S}{\prod_{i\in
      I}F_i}$. There are isomorphisms of $S$-modules
  \begin{equation*}
    \tp{S}{\prod_{i\in I}F_i} \dis 
    \tp{S}{\prod_{i\in I}(\tp[S]{S}{F_i})} \dis 
    \tp[S]{(\tp{S}{S})}{\prod_{i\in I}F_i} \dis
    \prod_{i\in I}F_i\:,
  \end{equation*}
  where the last isomorphism holds as $R\to S$ is an epimorphism of
  rings. Thus, $\prod_{i\in I}F_i,$ is a flat $S$-module, and it
  follows that $S$ is coherent; see \cite[Theorem 2.3.2]{ccr}.

  To see that $d$-perfectness ascends, let $R\to S$ be a flat
  epimorphism and assume that $R$ is $d$-perfect. Let $F$ be a flat
  $S$-module; it is also flat over $R$, so there is an exact sequence
  of $R$-modules $0\to P_d\to \cdots\to P_1\to P_0\to F\to 0$ with
  each $P_i$ projective. As $S$ is flat over $R$, it induces an exact
  sequence of $S$-modules
  \begin{equation*}
    0 \lra \tp{S}{P_d}
    \lra \cdots \lra \tp{S}{P_1} \lra \tp{S}{P_0} \lra
    \tp{S}{F} \lra 0\:.
  \end{equation*}
  Each $S$-module $\tp{S}{P_i}$ is projective, so one has
  $\pd[S]{(\tp{S}{F})} \le d$. Finally, as $R\to S$ is an epimorphism
  one has $\tp{S}{F} \is F$, so $d$-perfectness ascends. To prove
  descent, let $R\to S$ be a faithfully flat ring homomorphism where
  $S$ is $d$-perfect. Let $F$ be a flat $R$-module, and consider a
  projective resolution $\cdots \to P_1\to P_0\to F\to 0$ over $R$. As
  above it yields a projective resolution over $S$,
  \begin{equation*}
    \cdots \lra \tp{S}{P_d} \lra \cdots \lra
    \tp{S}{P_1} \lra \tp{S}{P_0} \lra \tp{S}{F} \lra 0\:.
  \end{equation*}
  The inequality $\pd[S]{(\tp{S}{F})} \le d$ implies that the
  $S$-module
  \begin{equation*}
    \Coker{(\tp{S}{P_{d+1}} \lra \tp{S}{P_d})} \dis \tp{S}{\Coker{(P_{d+1}\to P_d)}}
  \end{equation*}
  is projective, and it follows from \cite{P} that the $R$-module
  $\Coker{(P_{d+1}\to P_d)}$ is projective, whence one has $\pd{F} \le
  d$.
\end{proof}

\begin{theorem}
  \label{totac.is.Ntotac}
  Let $X$ be locally coherent and locally $d$-perfect scheme, and let
  $\cx{\P}$ be a complex of vector bundles on $X$. If there exists an
  open affine covering $\calU$ of $X$ such that $\cx{\P}(U)$ is
  totally acyclic in $\mathrm{Ch}(\mathscr{R}(U))$ for every $U \in
  \calU$, then $\cx{\P}$ is \textbf{F}-totally acyclic.
\end{theorem}

\begin{proof}
  Over a coherent $d$-perfect ring, every totally acyclic complex is
  \textbf{F}-totally acyclic; see Holm~\cite[Proposition
  3.4]{holm04}. Thus, every complex $\cx{\P}(U)$ is \textbf{F}-totally
  acyclic and the result follows as that is a Zariski-local property
  by Corollary~\ref{N-tot.acyc.def}.
\end{proof}

\begin{theorem}
  \label{tot.acyclic.mod}
  Let $X$ be locally coherent and $\cx{\P}$ be a complex of possibly
  infinite dimensional vector bundles. If $\cx{\P}$ is
  \textbf{F}-totally acyclic then $\cx{\P}(U)$ is totally acyclic for
  every open affine subset $U\subseteq X$.
\end{theorem}

\begin{proof}
  Let $U\subseteq X$ be an open affine subset. The complex
  $\cx{\P}(U)$ is an \textbf{F}-totally acyclic complex of projective
  $\R(U)$-modules. Since $\R(U)$ is a coherent ring, $\cx{\P}(U)$ is
  totally acyclic by a result of Bravo, Gillespie, and Hovey
  \cite[Theorem~6.7]{BGH}.
\end{proof}

For a different proof of the theorem, one could verify that the proof
of \cite[Lemma~4.20(ii)]{MS} extends to coherent rings.

\begin{remark}
  \label{rmk:tac}
  Let $X$ be a scheme with structure sheaf $\R$. A complex $\cx{\P}$
  of vector bundles on $X$ would be called totally acyclic if for
  every open affine subset $U \subseteq X$ the $\R(U)$-complex
  $\cx{\P}(U)$ is totally acyclic as defined in \ref{dfn:tot}. We have
  not explicitly addressed that property, because we do not know if it
  is Zariski-local. However, if $X$ is locally coherent and locally
  $d$-perfect, then the property is Zariski-local. Indeed, assume that
  there exists an open affine covering $\calU$ of $X$ such that
  $\cx{\P}(U)$ is a totally acyclic complex of projective
  $\R(U)$-modules for every $U\in\calU$. It follows from Theorem
  \ref{totac.is.Ntotac} that $\cx{\P}$ is \textbf{F}-totally acyclic,
  and then for every open affine subset $U\subseteq X$ the complex
  $\cx{\P}(U)$ is totally acyclic by Theorem~\ref{tot.acyclic.mod}.
\end{remark}

\section{Existence of adjoints}
\label{sec:adj}

\begin{definition}
  Let $X$ be a scheme; by $\Ntot$ we denote the class of
  \textbf{F}-totally acyclic complexes of quasi-coherent sheaves on
  $X$.
\end{definition}

\begin{theorem}
  \label{Ntot.complete}
  Let $X$ be a scheme. The class $\Ntot$ is covering in the category
  $\Ch{\Qcoh(X)}$, and if $\Qcoh(X)$ has a flat generator, then every
  such cover is an epimorphism.
\end{theorem}

The assumption about existence of a flat generator for $\Qcoh(X)$ is
satisfied if the scheme $X$ is quasi-compact and semi-separated; see
Alonso Tarr\'{\i}o, Jerem\'{\i}as L\'opez, and Lipman
\cite[(1.2)]{AJL1}. We prepare for the proof with a couple of lemmas.

\begin{lemma}
  \label{prev}
  Let $X$ be a scheme, $\cx{\M}$ be a complex in $\Ntot$, and
  $\cx{\M}'$ be a subcomplex of $\cx{\M}$. If conditions (1) and (2)
  below are satisfied, then the complexes $\cx{\M}'$ and
  $\cx{\M}/\cx{\M}'$ belong to $\Ntot$.
  \begin{enumerate}
  \item $\cx{\M}'$ is acyclic.
  \item There exists an open covering $\calU$ of $X$ such that the
    $\R(U)$-submodules $\M'_n(U)\subseteq \M_n(U)$ and
    $\Cy[n]{\cx{\M}'(U)} \subseteq \Cy[n]{\cx{\M}(U)}$ are pure for
    every $U\in\calU$ and all $n\in\ZZ$.
  \end{enumerate}
\end{lemma}

\begin{proof}
  Fix $U\in \calU$. For every $n\in\ZZ$ the submodule
  $\M'_n(U)\subseteq \M_n(U)$ is pure and $\M_n(U)$ is flat, so the
  modules $\M'_n(U)$ and $(\M_n/\M_n')(U)$ are flat.  By assumption,
  there is for every $n\in\ZZ$ a commutative diagram with exact rows
  \begin{equation*}
    \tag{$\ast$}
    \begin{gathered}
      \xymatrix{ 0 \ar[r] & \Cy[n]{\cx{\M}'(U)} \ar[r] \ar[d] &
        \M_n'(U) \ar[r] \ar[d]
        & \Cy[n-1]{\cx{\M}'(U)}  \ar[d] \ar[r] & 0 \\
        0 \ar[r] & \Cy[n]{\cx{\M}(U)} \ar[r] & \M_n(U) \ar[r] &
        \Cy[n-1]{\cx{\M}(U)} \ar[r] & 0\:, }
    \end{gathered}
  \end{equation*}
  where the vertical homomorphisms are pure embeddings. Let $I$ be an
  injective $\R(U)$-module. In the commutative diagram obtained by
  applying $\tp[\R(U)]{I}{-}$ to $(\ast)$, the vertical homomorphisms
  are pure embeddings, and the bottom row is exact as the complex
  $\cx{\M}(U)$ is \textbf{F}-totally acyclic. By commutativity it
  follows that the homomorphism $\tp{I}{\Cy[n]{\cx{\M}'(U)}} \to
  \tp{I}{\cx{\M}'(U)}$ is injective, whence
  $\tp[\R(U)]{I}{\cx{\M}'(U)}$ is acyclic; i.e.\ $\cx{\M}'(U)$ is
  \textbf{F}-totally acyclic. Now, as $(\cx{\M}/\cx{\M}')(U)$ is
  a complex of flat $\R(U)$-modules, the sequence
  \begin{equation*}
    0 \lra \tp[\R(U)]{I}{\cx{\M}'(U)} \lra \tp[\R(U)]{I}{\cx{\M}(U)}
    \lra \tp[\R(U)]{I}{(\cx{\M}/\cx{\M}')(U)} \lra 0
  \end{equation*}
  is exact for every injective $\R(U)$-module $I$; since the left-hand
  and middle complexes are acyclic, so is the right-hand complex, that is,
  $(\cx{\M}/\cx{\M}')(U)$ is \textbf{F}-totally acyclic. Finally,
  since $U\in \calU$ is arbitrary, it follows that $\cx{\M}'$ and
  $\cx{\M}/\cx{\M}'$ are complexes of flat quasi-coherent sheaves and
  \textbf{F}-totally acyclic per Corollary \ref{N-tot.acyc.def}.
\end{proof}

For the next lemma, recall from \ref{ss:pure} the notion of a
$\kappa$-pure morphism.

\begin{lemma}
  \label{sections.pure}
  Let $X$ be a scheme and $\kappa$ be a cardinal.  If
  $\mapdef{\tau}{\F}{\G}$ is a $\kappa$-pure morphism in $\Qcoh(X)$,
  then $\mapdef{\tau(U)}{\F(U)}{\G(U)}$ is pure monomorphism of
  $\R(U)$-modules for every open affine subset $U\subseteq X$.
\end{lemma}

\begin{proof}
  The category $\Qcoh(X)$ is a Grothendieck category; see
  \cite[Corollary 3.5]{EE}. By a result of Beke \cite[Proposition
  3]{Beke} there exists an infinite cardinal $\kappa$, such that
  $\Qcoh(X)$ is locally $\kappa$-presentable. The pure morphism $\tau$
  is by \cite[Proposition 2.29]{AR} a monomorphism, so it yields a
  $\kappa$-pure exact sequence $$\mathbb{E}\deq 0\lra \F\lra \G\lra
  \H\lra 0\:.$$ By \cite[Proposition 2.30]{AR} it is the colimit of a
  $\kappa$-directed system $(\mathbb{E}_{\alpha})$ of splitting short
  exact sequences in $\Qcoh(X)$. For every open affine subset
  $U\subseteq X$, one gets a $\kappa$-directed system
  $(\mathbb{E}_{\alpha}(U))$ of split exact sequences of
  $\R(U)$-modules. Since $\kappa$ is infinite, every $\kappa$-directed
  system is an $\aleph_0$-directed system; see for example \cite[Fact
  A.2]{G}. Therefore, $\mathbb{E}(U)$ is a direct limit of splitting
  short exact sequences of $\R(U)$-modules. Hence, the sequence
  $\mathbb{E}(U)$ of $\R(U)$-modules is pure.
\end{proof}

\begin{lemma}
  \label{Ntot.Kap}
  For every scheme $X$, the class $\Ntot$ is a Kaplansky class.
\end{lemma}

\begin{proof}
  Let $\kappa$ be a cardinal such that the categories $\Qcoh(X)$ and
  $\Ch{\Qcoh(X)}$ are locally $\kappa$-presentable; see
  \cite[Corollary 3.5]{EE}, \cite[Proposition 3]{Beke}, and
  \ref{ss:complexes}. Let $\cx{\M} \ne 0$ be a complex in
  $\Ch{\Qcoh(X)}$; by \cite[Theorem~2.33]{AR} there is a cardinal
  $\gamma$ such that every $\gamma$-presentable subcomplex $\cx{\M}''
  \subseteq \cx{\M}$ can be embedded in a $\gamma$-presentable
  $\kappa$-pure subcomplex $\cx{\M'} \subseteq \cx{\M}$.  To prove
  that $\Ntot$ is a Kaplansky class, it is enough to verify that
  $\cx{\M}'$ satisfies conditions (1) and (2) in
  Lemma~\ref{prev}. Now, for any $\kappa$-presentable sheaf $\L \in
  \Qcoh(X)$, the complexes $\Scx{n}{\L}$ and $\Dcx{n}{\L}$ are
  $\kappa$-presentable. It follows that the $\kappa$-pure exact
  sequence $0 \to \cx{\M}' \to \cx{\M} \to \cx{\M}/\cx{\M}' \to 0$
  induces exact sequences
  \begin{equation*}
    0 \to \Mor(\Scx{n}{\L},\cx{\M}') \to \Mor(\Scx{n}{\L},\cx{\M}) 
    \to \Mor(\Scx{n}{\L},\cx{\M}/\cx{\M}') \to 0
  \end{equation*}  
  and
  \begin{equation*}
    0 \to \Mor(\Dcx{n}{\L},\cx{\M}') \to \Mor(\Dcx{n}{\L},\cx{\M}) 
    \to \Mor(\Dcx{n}{\L},\cx{\M}/\cx{\M}') \to 0\:,
  \end{equation*}
  where $\Mor$ is short for $\Hom_{\Ch{\Qcoh(X)}}$. For every
  $\cx{\N}\in\Ch{\Qcoh(X)}$ and $n\in\ZZ$ there are standard
  isomorphisms
  \begin{equation*}
    \Mor(\L,\Cy[n]{\cx{\N}}) \dis \Mor(\Scx{n}{\L},\cx{\N}) \quad\text{and}\quad
    \Mor(\L,\N_n) \dis \Mor(\Dcx{n}{\L},\cx{\N})\:,
  \end{equation*}
  which allow us to rewrite the exact sequences above as
  \begin{equation*}
    0 \to \Mor(\L,\Cy[n]{\cx{\M}'}) \to \Mor(\L,\Cy[n]{\cx{\M}}) 
    \to \Mor(\L,\Cy[n]{\cx{\M}/\cx{\M}'}) \to 0
  \end{equation*}  
  and
  \begin{equation*}
    0 \to \Mor(\L,\M_n') \to \Mor(\L,\M_n) 
    \to \Mor(\L,\M_n/\M_n') \to 0\:.
  \end{equation*}
  As the category $\Qcoh(X)$ is locally $\kappa$-presentable, it has a
  generating set of $\kappa$-presentable objects; it follows that $0
  \to \Cy[n]{\cx{\M}'} \to \Cy[n]{\cx{\M}} \to
  \Cy[n]{\cx{\M}/\cx{\M}'} \to 0$ and $0 \to \M'_n \to \M_n \to
  \M_n/\M'_n \to 0$ are exact, even $\kappa$-pure exact, sequences in
  $\Qcoh(X)$. Let $U \subseteq X$ be an open affine subset; by
  Lemma~\ref{sections.pure}, the sequences of $\R(U)$-modules $0 \to
  \Cy[n]{\cx{\M}'(U)} \to \Cy[n]{\cx{\M}(U)} \to
  \Cy[n]{(\cx{\M}/\cx{\M}')(U)} \to 0$ and \mbox{$0 \to \M'_n(U) \to
    \M_n(U) \to (\M_n/\M'_n)(U) \to 0$} are pure exact. Thus,
  condition $(2)$ in Lemma~\ref{prev} is satisfied. It remains to show
  that the subcomplex $\cx{\M}'$ is acyclic. To this end, apply the
  snake lemma to the canonical diagram
  \begin{equation*}
    \xymatrix{
      0 \ar[r] & \Cy[n]{\cx{\M}'} \ar[r] \ar[d] & \Cy[n]{\cx{\M}} \ar[r] \ar[d] &
      \Cy[n]{\cx{\M}/\cx{\M}'} \ar[r] \ar[d] & 0\\
      0 \ar[r] & \M'_n \ar[r] & \M_n \ar[r] &\M_n/\M'_n \ar[r] & 0
    }
  \end{equation*}
  to get an exact sequence $0 \to \Bo[n-1]{\cx{\M}'} \to
  \Bo[n-1]{\cx{\M}} \to \Bo[n-1]{\cx{\M}/\cx{\M}'} \to 0$ for every
  $n\in\ZZ$. Now apply the snake lemma to
  \begin{equation*}
    \xymatrix{
      0 \ar[r] & \Bo[n]{\cx{\M}'} \ar[r] \ar[d] & \Bo[n]{\cx{\M}} \ar[r] \ar[d] &
      \Bo[n]{\cx{\M}/\cx{\M}'} \ar[r] \ar[d] & 0\\
      0 \ar[r] & \Cy[n]{\cx{\M}'} \ar[r] & \Cy[n]{\cx{\M}} \ar[r] &
      \Cy[n]{\cx{\M}/\cx{\M}'} \ar[r] & 0
    }
  \end{equation*}
  One has $\Bo[n]{\cx{\M}} = \Cy[n]{\cx{\M}}$ for every $n\in\ZZ$ as
  $\cx{\M}$ is acyclic. It follows that $\Bo[n]{\cx{\M'}} =
  \Cy[n]{\cx{\M'}}$ holds for all $n\in\ZZ$, so $\cx{\M'}$ is acyclic.
\end{proof}

We can prove now Theorem \ref{Ntot.complete}.

\begin{proof}[Proof of Theorem \ref{Ntot.complete}]
  The goal is to apply Proposition \ref{Kap.covering}. The class
  $\Ntot$ is Kaplansky by Lemma~\ref{Ntot.Kap}, and it remains to
  prove that it is closed under extensions and direct limits. Let
  $0\to \cx{\M}' \to \cx{\M} \to \cx{\M}'' \to 0$ be an exact sequence
  in $\Ch{\Qcoh(X)}$ with $\cx{\M}'$ and $\cx{\M}''$ in $\Ntot$. For
  every open affine subset $U\subseteq X$, there is an exact sequence
  $0\to \cx{\M}'(U)\to \cx{\M}(U)\to \cx{\M}''(U)\to 0$ of complexes
  of $\R(U)$-modules. As $\cx{\M}'$ and $\cx{\M}''$ are complexes of
  flat modules, so is $\cx{\M}$. The sequence remains exact when
  tensored by an injective $\R(U)$-module, and since $\cx{\M'}(U)$ and
  $\cx{\M}''(U)$ are \textbf{F}-totally acyclic, so is
  $\cx{\M}(U)$. It follows that $\cx{\M}$ belongs to $\Ntot$; that is,
  $\Ntot$ is closed under extensions.

  Let $\setof{\cx{\F}^i}{i\in I}$ be a direct system of complexes in
  $\Ntot$.  For every open affine subset $U\subseteq X$ one then has a
  direct system $\setof{\cx{\F}^i(U)}{i\in I}$ of \textbf{F}-totally
  acyclic complexes of $\R(U)$-modules. Now, $\varinjlim_{i\in I}
  \cx{\F}^i(U)$ is an acyclic complex of flat $\R(U)$-modules, and
  since direct limits commute with tensor products and homology, it is
  \textbf{F}-totally acyclic. The quasi-coherent sheaf $\cx{\F} =
  \varinjlim_{i\in I} \cx{\F}^i$ satisfies $\cx{\F}(U) =
  \varinjlim_{i\in I} \cx{\F}^i(U)$, so we conclude that $\cx{\F}$
  belongs to $\Ntot$. Now it follows from Proposition
  \ref{Kap.covering} that $\Ntot$ is covering.

  Now assume that the category $\Qcoh(X)$ has a flat generator
  $\F$. It follows that the complexes of flat sheaves $\Dcx{n}{\F}$,
  $n\in\ZZ$, generate the category $\Ch{\Qcoh(X)}$. Evidently, each
  complex $\Dcx{n}{\F}$ is \textbf{F}-totally acyclic, so every
  $\Ntot$-cover is an epimorphism.
\end{proof}

For any scheme $X$, one can consider the homotopy category of flat
quasi-coherent sheaves, which we denote $\Kflat$. It is a triangulated
category, and the full subcategory $\Ktot$ of \textbf{F}-totally
acyclic complexes in $\Kflat$ is a triangulated subcategory. The next
result generalizes \cite[Corollary 4.26]{MS} to arbitrary schemes.

\begin{corollary}
  \label{cor.right.adj}
  For any scheme $X$, the inclusion $\Ktot\to \Kflat$ has a right
  adjoint.
\end{corollary}

\begin{proof}
  The full subcategory $\Ktot$ is closed under retracts, so per Neeman
  \cite[Definition 1.1]{Nee2} it is a thick subcategory of
  $\Kflat$. By Theorem~\ref{Ntot.complete} every complex in $\Kflat$
  has a $\Ktot$-precover; now \cite[Proposition 1.4]{Nee2} yields the
  existence of a right adjoint to the inclusion $\Ktot\to \Kflat$.
\end{proof}

\begin{definition}
  Let $X$ be a scheme. A quasi-coherent sheaf $\G$ on $X$ is called
  \emph{Gorenstein flat} if there exists an \textbf{F}-totally acyclic
  complex $\cx{\F}$ of flat quasi-coherent sheaves on $X$ with $\G\is
  \Cy[n]{\cx{\F}}$ for some $n\in \ZZ$.
\end{definition}

The affine case of the next result was proved by Yang and Liang
\cite[Theorem~A]{yang14}. Recall from \cite[(1.2)]{AJL1} that
$\Qcoh(X)$ has a flat generator---in particular, a Gorenstein flat
generator---if the scheme $X$ is quasi-compact and semi-separated.

\begin{corollary}
  \label{Gor.flat.prec}
  Let $X$ be a scheme. Every quasi-coherent sheaf on $X$ has a
  Gorenstein flat precover. Moreover, if the category $\Qcoh(X)$ has a
  Gorenstein flat generator, then the Gorenstein flat precover is an
  epimorphism.
\end{corollary}

\begin{proof}
  Let $\M\in \Qcoh(X)$; by Theorem \ref{Ntot.complete} there exists an
  \textbf{F}-totally acyclic cover $\varphi\colon\cx{\F}\to
  \Scx{1}{\M}$ in $\Ch{\Qcoh(X)}$. For every acyclic complex
  $\cx{\mathscr A}$ in $\Ch{\Qcoh(X)}$ there is an isomorphism
  \begin{equation*}
    \tag{$\ast$}
    \Hom_{\Ch{\Qcoh(X)}}(\cx{\mathscr A}, \Scx{1}{\M}) 
    \dis \Hom_{\Qcoh(X)}(\Cy[0]{\cx{\mathscr A}},\M)\:.
  \end{equation*}
  Thus, the cover induces a morphism $\phi\colon\Cy[0]{\cx{\F}}\to
  \M$, where $\Cy[0]{\cx{\F}}$ is Gorenstein flat.  We now argue that
  $\phi$ is a Gorenstein flat precover. Consider a morphism
  $\theta\colon\mathscr G\to \M$ in $\Qcoh(X)$ with $\mathscr G$
  Gorenstein flat. There exists an \textbf{F}-totally acyclic complex
  $\cx{\F}'$ of flat quasi-coherent sheaves with $\Cy[0]{\cx{\F}'} \is
  \mathscr G$. The morphism $\theta$ corresponds by $(\ast)$ to a
  morphism $\vartheta\colon \cx{\F}'\to \Scx{1}{\M}$ in
  $\Ch{\Qcoh(X)}$. Since $\varphi$ is an \textbf{F}-totally acyclic
  cover, there exists a morphism $\varkappa\colon\cx{\F}'\to \cx{\F}$
  such that $\varphi\circ\varkappa=\vartheta$ holds. It is now
  straightforward to verify that the induced morphism $\kappa\colon
  \Cy[0]{\cx{\F}'} \to \Cy[0]{\cx{\F}}$ satisfies $\phi\circ \kappa =
  \theta$. Thus $\phi$ is a Gorenstein flat precover of $\M$. Finally,
  a flat generator of $\Qcoh(X)$ is trivially Gorenstein flat, so if
  such a generator exists, then every Gorenstein flat precover is
  surjective.
\end{proof}

The next corollary partly recovers \cite[Proposition~8.10]{BGH} and
\cite[Corollary~2]{EIO}; the argument is quite different from the one
given in \cite{BGH} and \cite{EIO}.

\begin{corollary}
  \label{cor.Gor.prec}
  Let $R$ be a coherent and $d$-perfect ring. Every $R$-module has a
  Gorenstein projective precover.
\end{corollary}

\begin{proof}
  Let $\ProjR$ denote the class of projective $R$-modules. In the
  category of complexes of $R$-modules,
  $(\Ch{\ProjR},\Ch{\ProjR}^{\perp})$ is a complete cotorsion pair;
  see Gillespie \cite[Section 5.2]{G08}. Let $F$ be a complex in
  $\Ch{\FlatR}$; there is an exact sequence
  \begin{equation*}
    \tag{$*$}
    0\lra E\lra P\lra F\lra 0
  \end{equation*} 
  with $P$ is in $\Ch{\ProjR}$ and $E$ in $\Ch{\ProjR}^\perp$. As the
  class of flat $R$-modules is resolving, $E$ belongs to the
  intersection $\Ch{\ProjR}^{\perp}\cap \Ch{\FlatR}$, which by Neeman
  \cite[2.14]{Nee} is the class of pure acyclic complexes of flat
  $R$-modules. Notice that $P$ is \textbf{F}-totally acyclic if and
  only if $F$ is \textbf{F}-totally acyclic. It thus follows from
  Theorem \ref{tot.acyclic.mod} that every \textbf{F}-totally acyclic
  complex has a totally acyclic precover.

  Let $M$ be any complex of $R$-modules. By Theorem
  \ref{Ntot.complete} it has an \textbf{F}-totally acyclic cover
  $\mapdef{\phi}{F}{M}$, and as noted above $F$ has a totally acyclic
  precover $\mapdef{\psi}{P}{F}$. We argue that the composite
  $\phi\circ\psi$ is a totally acyclic precover of $M$. To this end,
  let $\delta\colon P'\to M$ be a morphism, where $P'$ is a totally
  acyclic complex; by Theorem \ref{totac.is.Ntotac} it is
  \textbf{F}-totally acyclic. Since $\phi$ is an \textbf{F}-totally
  acyclic cover of $M$, there exists a morphism $\theta\colon P'\to F$
  with $\phi\circ \theta = \delta$. Now, since $\psi$ is a totally
  acyclic precover of $F$, there exists $\g\colon P'\to P$ such that
  $\psi\circ \g = \theta$. Finally, one has $(\phi\circ\psi)\circ \g =
  \phi\circ\theta = \delta$. So $\phi\circ\psi$ is a totally acyclic
  precover of $M$. Finally, arguing as in the proof of Corollary
  \ref{Gor.flat.prec}, we infer that every module has a Gorenstein
  projective precover.
\end{proof}

\section*{Acknowledgments}

\noindent
We thank the anonymous referees for numerous comments and suggestions that
helped us clarify the exposition. We are particularly grateful to one referee
for a correspondence that led to Proposition \ref{prp:equiv_qcss}.

\end{document}